\documentclass[10pt, twocolumn, conference]{IEEEtran}
\makeatletter
\def\ps@headings{%
\def\@oddhead{\mbox{}\scriptsize\rightmark \hfil \thepage}%
\def\@evenhead{\scriptsize\thepage \hfil \leftmark\mbox{}}%
\def\@oddfoot{}%
\def\@evenfoot{}}
\makeatother
\usepackage{amsmath, mathrsfs,amsfonts}
\usepackage{amsfonts}
\usepackage{amssymb}
\usepackage{acronym}
\usepackage{graphicx}
\usepackage{cite}
\usepackage{algorithmic}
\usepackage{algorithm2e}
\usepackage{array}
\usepackage{hyperref}
\usepackage{subfigure}

\def\ignore#1{}

\def\S{{\cal S}}

\def \pr  {\mathbf{Pr}}

\def \OPT {\mathit{\xi}}
\def \RISKY {\textit{$\varepsilon$-Survivable }}
\def \epsnet {\textit{$\varepsilon$-net }}

\newtheorem{definition}{Definition}

\newtheorem{theorem}{Theorem}
\newtheorem{lemma}{Lemma}
\newtheorem{corollary}{Corollary}

               {\begin{list}{}{\leftmargin#1\rightmargin#2\topsep#3}\item{}}%
               {\end{list}}

\IEEEoverridecommandlockouts
\begin{document}
\title{Survivable Paths in Multilayer Networks\thanks{This work was supported by NSF grants CNS-0830961 and CNS-1017800, and by DTRA grants HDTRA1-07-1-0004 and HDTRA-09-1-005.}}

\author{ Marzieh Parandehgheibi, Hyang-Won Lee and Eytan Modiano\\
\normalsize{Laboratory for Information and Decision Systems}\\
\normalsize{Massachusetts Institute of Technology}\\
\normalsize{Cambridge, MA 02139}}

\maketitle \thispagestyle{headings}

\begin{abstract}

We consider protection problems in multilayer networks. In single-layer networks, a pair of disjoint paths can be used to provide protection for a source-destination pair. However, this approach cannot be directly applied to layered networks where disjoint paths may not always exist. In this paper, we take a new approach which is based on finding a \emph{set of paths} that may not be disjoint but together will survive any single physical link failure. First, we consider the problem of finding the minimum number of survivable paths. In particular, we focus on two versions of this problem: one where the length of a path is restricted, and the other where the number of paths sharing a fiber is restricted. We prove that in general, finding the minimum survivable path set is NP-hard, whereas both of the restricted versions of the problem can be solved in polynomial time. We formulate the problem as Integer Linear Programs (ILPs), and use these formulations to develop heuristics and approximation algorithms. Next, we consider the problem of finding a set of survivable paths that uses the minimum number of fibers. We show that this problem is NP-hard in general, and develop heuristics and approximation algorithms with provable approximation bounds. Finally, we present simulation results comparing the different algorithms.

\end{abstract}

\section{Introduction}
Multilayer network architecture such as IP-over-WDM has played an important role in advancing modern communication networks. Typically, a layered network is constructed by embedding a logical topology onto a physical topology such that each logical link is routed using a path in the physical topology. While such a layering approach enables to take advantage of the flexibility of upper layer technology (e.g., IP) and the high data rates of lower layer technology (e.g., WDM), it raises a number of challenges for efficient and reliable operations. In this paper, we focus on the issue of providing protection in layered networks.

The protection problem in single-layer networks is rather straightforward; namely, providing a pair of disjoint paths (one for primary and one for backup) guarantees a route between two nodes against any single link failure. This approach, however, cannot be directly applied to layered networks, because a pair of seemingly disjoint paths at the logical layer may share a physical link and thus simultaneously fail in the event of a physical link failure. To address this issue, \cite{Bhandari} introduced the notion of \emph{physically disjoint} logical paths.

In \cite{Hu}, this notion was generalized as Shared Risk Link Group (SRLG) disjoint paths, i.e., two paths between the source and destination nodes that do not share any risk (e.g., fiber and conduit). Nearly all the previous works in the context of layered network protection have focused on finding SRLG-disjoint paths~\cite{TrapAvoidance},~\cite{PROMISE},~\cite{Datta},~\cite{Xu}.

Although the SRLG-disjoint paths problem has been well studied, there are several challenges to this approach. First, SRLG-disjoint paths may not always exist. Second, such a pair of paths could be very long and thus vulnerable. Third, by associating appropriate cost to a path, the SRLG-disjoint paths problem can be modified to find a path set avoiding long paths. However, the modified problem is only known to be NP-hard~\cite{Hu} and there is no known algorithm with provable approximation guarantee. 

In order to address these challenges, we take an alternative approach that is based on finding a set of paths that together will survive any single physical link failure. Thus, in the case that SRLG-disjoint paths do not exist, we may find three or more paths such that in the event of a fiber failure, at least one of the paths remain connected. This notion of \emph{survivable path set} generalizes the traditional notion of SRLG-disjoint paths, and enables to provide protection for a broader range of scenarios. Our contributions can be summarized as follows:

\begin{itemize}
\item We introduce a new notion of survivable path set so as to provide protection even for the case where SRLG-disjoint paths do not exist.
\item We prove the NP-hardness of various protection problems that seek to find a survivable path set with different objectives, and identify the conditions for the protection problems to become polynomial time solvable.
\item We develop heuristics and approximation algorithms for the survivable path set problems.
\end{itemize}

In Section II, we present the network model. In Section III, we study the problem of finding a minimum set of paths that will survive any single fiber failure and develop several approximation algorithms. In Section IV, we design approximation algorithms for finding a survivable path set that uses the minimum number of fibers. Finally, we provide simulation results in Section V and conclusions in Section VI.

\section{Network Model}
We consider a layered network that consists of a logical topology $G_{L}=(V_{L}, E_{L})$ built on top of a physical topology $G_{P}=(V_{P}, E_{P})$ where $V$ and $E$ are the sets of nodes and links respectively. Each logical link $(i,j)$ in $E_{L}$ is mapped onto an $i-j$ path in the physical topology. This is called lightpath routing. Different lightpaths may use the same fiber (physical link), therefore when a fiber fails, all the lightpaths using that fiber will fail. Hence, a logical path survives the failure of any fiber that it does not use.

As mentioned above, we generalize the traditional notion of SRLG-disjoint paths to account for the case where there does not exist a pair of SRLG-disjoint paths. In a layered network, a set of logical paths is said to be \emph{survivable} if at least one of the paths remain connected after any single physical link failure. Hence, a survivable set consisting of two paths is a pair of SRLG-disjoint paths. Note that there may exist a survivable path set while SRLG-disjoint paths do not exist. For example, consider the physical and logical topologies in Fig. \ref{topologies}. Each dashed line in Fig. \ref{fig:MAPPING_MSP} shows the lightpath routing of each logical link over the physical topology. Under this lightpath routing, each pair of logical paths between nodes $1$ and $4$ shares some fiber.

Suppose that we want to find a set of logical paths between nodes $1$ and $4$ in Fig. \ref{topologies} that can survive any single physical link failure. Clearly, there does not exist a pair of SRLG-disjoint paths as each pair of logical paths shares a fiber. However, it is straightforward to check that the set of 3 paths can survive any single fiber cut, although they are not SRLG-disjoint. This example shows that the traditional protection schemes based on SRLG-disjoint paths (such as the ones in~\cite{Hu}) may fail to provide protection against single physical link failures, while there exists a set of paths that can together provide protection. Our goal in this paper is to address the problem of finding a set of survivable paths that together will survive any single fiber failure.

\begin{figure}[h]
\centering
\subfigure[Physical Topology]
{\label{fig:physical_MSP}\includegraphics[scale=0.40]{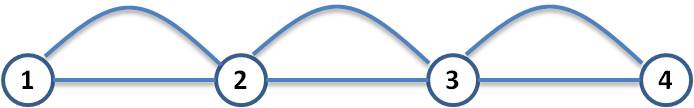}}                
\subfigure[Logical Topology]
{\label{fig:logical_MSP}\includegraphics[scale=0.40]{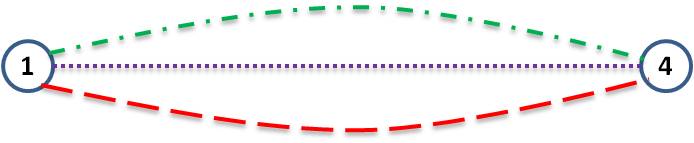}}
\subfigure[Mapping]
{\label{fig:MAPPING_MSP}\includegraphics[scale=0.40]{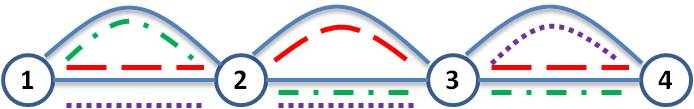}}
\caption{Topologies in Multilayer Networks}
\label{topologies}
\end{figure}

\section{Minimum Survivable Paths Set (MSP)}\label{MSP_SEC}
We start with the problem of finding a minimum survivable path set, i.e., the minimum cardinality set of paths between a pair of nodes $s$ and $t$ that survive any single physical link (fiber) failure. We first present a path-based Integer Linear Program (ILP) formulation for this problem, assuming that the entire set of $s-t$ paths with their routings over fibers is given. For each path $j$, let $P_j$ be a binary variable which takes the value 1 if path $j$ is selected, and 0 otherwise. The matrix $A \in R ^{m \times n}$ refers to the mapping of all $n$ paths over the $m$ fibers such that  $a_{ij}=0$ if path $j$ uses fiber $i$ and $a_{ij}=1$ otherwise. Let $e$ be a $m \times 1$ vector of ones.
\begin{align}
\mbox{minimize} \quad &\sum_{j=1}^n P_{j}\label{eqn:msp-obj} \\
\mbox{subject to} \quad &A \times P \geq e  \label{survivable_const} \\
 \quad &P_{j} \in \{0,1\}, \quad j=1,\cdots,n \label{eqn:msp-bin}
\end{align}

In the above, the objective function is the number of selected paths. Each row $i \in \{1,\cdots,m\}$ in constraint (\ref{survivable_const}) requires that at least one selected path survives the failure of fiber $i$, i.e., the selected path set should be survivable. Hence, the optimal solution to the above optimization problem gives a minimum survivable path set.
Although this formulation requires the knowledge of every path (which is possibly exponential in the number of fibers), the compact and clean expression of the path-based formulation enables us to analyze the useful properties of survivable path sets. Later, we will use this formulation to develop heuristics and approximation algorithms for finding a minimum survivable path set.

The MSP problem can also be formulated using a polynomial number of constraints and variables without enumerating all of the paths. Let $P_{tot}$ denote the number of selected $s-t$ logical paths, $E_{L}$ denote the set of logical links and $E_L^k$ denote the set of remaining logical links after the failure of fiber $k$. Note that for survivability, each $E_L^k$ should contain at least one of the selected paths. Let $x_{ijk}$ be 1 if link $(i,j)$ in $E_{L}^{k}$ is selected to form an $s-t$ path over the remaining graph $G_L^k=(V_L, E_L^k)$, and 0 otherwise. Let $y_{ij}$ be 1 if the selected path set uses logical link $(i,j)$, and 0 otherwise. The following link-based formulation describes the MSP problem.
\begin{equation}
\begin{array}{ll}
\mbox{minimize} \quad \quad &  P_{tot} \quad \quad \quad \quad \quad \quad \quad \quad \quad \quad \quad \quad \quad \quad \quad 
\end{array}
\end{equation}
\begin{equation}
\left.\begin{aligned}\label{flow_main}
\mbox{subject to} \quad& \sum_{(s,j) \in E_{L}} y_{sj} = P_{tot} &\\
 \quad& \sum_{(i,t) \in E_{L}} y_{it} = P_{tot} &\\
 \quad& \sum_{(i,j) \in E_{L}} y_{ij} - \sum_{(j,i) \in E_{L}} y_{ji} = 0, &\forall i\neq s,t 
\end{aligned}
\right\}
\end{equation}
\begin{equation}
\left.\begin{aligned}\label{flow_remain}
\quad &  \sum_{(s,j) \in E_{L}^{k}} x_{sjk} = 1, \quad &\forall k\\
\quad &\sum_{(i,t) \in E_{L}^{k}} x_{itk} = 1, \quad &\forall k \\
\quad & \sum_{(i,j) \in E_{L}^{k}} x_{ijk} - \sum_{(j,i) \in E_{L}^{k}} x_{jik} = 0 \quad &\forall k, \forall i\neq s,t
\end{aligned}
\right\}
\end{equation}
\begin{align}
& y_{ij} \geq x_{ijk} & \forall k,i,j\label{link_fiber_force} \\
& x_{ijk} \in \{0,1\} & \forall k,i,j
\end{align}
The constraints in (\ref{flow_remain}) require that each remaining logical graph $E_{L}^{k}$ should contain an $s-t$ path, which guarantees the survivability against any single physical link failure. By the constraints in (\ref{link_fiber_force}), logical link $(i,j)$ is selected if it has been used in some remaining logical graph $G_L^k=(V_L,E_L^k)$. Hence, the constraints in (\ref{flow_main}) require that there should be total $P_{tot}$ flows between nodes $s$ and $t$ over the selected logical links specified by $y_{ij}$'s. Consequently, the variable $P_{tot}$ counts the total number of paths selected for survivability. In Section \ref{simulation_sec}, we will use this formulation to verify the performance bound of our approximation algorithms.

\subsection{MSP in general setting}\label{MSP_general}
In this section, we show that the MSP problem is NP-hard in general and discuss some algorithms that can be used to solve the problem. In Sections \ref{sub_K} and \ref{sub_WDM}, we will study the MSP problem under practical constraints. Our first result pertains to the complexity of the MSP problem as stated in Theorem \ref{MSP_complexity} below.

\begin{theorem}\label{MSP_complexity}
Computing the minimum number of survivable paths in multilayer networks is NP-hard. In addition, this minimum value cannot be approximated within any constant factor, unless $P=NP$.
\end{theorem}

The proof of Theorem \ref{MSP_complexity} relies on a mapping between the survivable path set problem and the minimum set cover problem. Suppose that each path corresponds to a set of fibers that are not used by that path, i.e., survived. Then, finding a minimum survivable path set is equivalent to finding a minimum path set that survives (covers) all of the fibers. The complete proof can be found in Appendix \ref{MSP_proof}.

Since the problem is computationally hard to solve, we consider heuristics and approximation algorithms that give a set of survivable paths in polynomial time. Owing to the similarity to the set cover problem, the heuristics that have been developed for set cover problems can be used here. In particular, a common approach to solve the set cover problem is the greedy algorithm. In order to apply the greedy algorithm to our setting, one needs to enumerate all of the paths with their routings on the fibers. In general, the number of paths in a multilayer network is exponential in the total number of fibers. Moreover, in each iteration, the greedy algorithm tries to find a path that survives the maximum number of fibers. This is equivalent to the Minimum Color Path problem, which is known to be NP-hard.~\cite{shortest_path}

Another approach which can be used to approximate the set cover problem is randomized rounding. Randomized rounding gives an $O(\log m)$ approximation, where $m$ is the number of fibers~\cite{Raghavan}. This is the best possible approximation for the MSP problem, which is due to the fact that the minimum set cover problem cannot be approximated within better than a $\log m$ factor~\cite{hardness}.

Fortunately, practical systems impose certain physical constraints that make the survivable path-set problem easier to solve. For example, due to physical impairments and delay constraints, paths are typically limited in length. Furthermore, in WDM networks, the sharing of a fiber by the logical links is limited by the number of available wavelengths. In the following, we show that these physical limitations make the MSP problem tractable.

\subsection{The Path Length Restricted Version}\label{sub_K}
In this section, we assume that each logical path is restricted to use at most $K$ fibers. Restricting the length of paths (i.e., number of fibers on each path) is a realistic assumption, because each logical link is typically constrained in the number of fibers that it may use, and each logical path is constrained in the number of logical links.

\begin{lemma}\label{K-OPT}
Under the path length restriction, the optimal number of survivable paths is at most $K+1$.
\end{lemma}
\begin{proof}
By the assumption, each path uses at most $K$ fibers, and thus at least $m-K$ fibers are survived by a path. Suppose that we have selected an arbitrary path, and want to add other paths to form a survivable path set. In the worst case, each of the newly selected paths can survive only a single fiber which is not survived by the previously selected paths. Since there are at most $K$ fibers that are not survived by the first path, we need at most $K$ additional paths to survive the rest of the fibers. Therefore, the total number of paths will not exceed $K+1$.
\end{proof}

\begin{lemma}\label{K-paths}
In the path length restricted version of MSP, the total number of paths is polynomial in the number of fibers $m$, and can be enumerated in polynomial time.
\end{lemma}
\begin{proof}
Under the assumption, a path can consist of up to $K$ fibers, and thus at most $K$ logical links. In a graph with $n$ nodes there can be $O(n^K)$ paths of length up to $K$. Since the number of nodes is at most $2m$, the total number of logical paths of length up to $K$ is $O(m^K)$. A simple exhaustive search can be used to enumerate the paths.
% In the physical topology, the maximum number of distinct paths (i.e. paths that use different set of fibers) with size $i$ is ${{m}\choose{i}}$. Therefore, total number of distinct paths with sizes $2,\cdots,K$ is $O(m^K)$. Although finding the set of distinct paths is enough for the purpose of our algorithms, it is straightforward to show that the total number of logical paths is also polynomial and can be found in polynomial time (See \cite{TechRep}).
%On the other hand, as many as $2^{i-2}$ different logical paths can traverse a physical path of length $i$. Therefore, the maximum number of logical paths using exactly $i$ fibers is $(2m)^i$ and the total number of logical paths with sizes $2,\cdots,K$ is $O(m^K)$.
\end{proof}

\begin{theorem}\label{Ksearch}
The path length restricted version of the MSP problem can be solved in polynomial time.
\end{theorem}
\begin{proof}
By Lemma~\ref{K-OPT}, MSP needs at most $K+1$ paths to survive any single failure. Therefore, one can find the exact solution by searching through all subsets of paths with sizes $2,3,...,K+1$. This will take $O(P^{K+1})$ iterations where $P$ is the total number of paths.
On the other hand, by Lemma~\ref{K-paths}, the total number of paths is $O(m^{K})$. Therefore, the total running time of exhaustive search is $O(m^{K(K+1)})$ which is polynomial in the total number of fibers.
\end{proof}

Although this exhaustive search returns an optimal solution, its running time can be prohibitive for large values of $m$ and $K$. This motivates us to study heuristics and approximation algorithms with better running time. First, we consider a greedy algorithm, followed by a randomized algorithm based on $\varepsilon$-net which is a well-known technique in the area of computational geometry.

\subsubsection{Greedy Algorithm}\label{MSPG-SEC}
The first heuristic we consider is a greedy algorithm which is similar to the greedy algorithm for the minimum set cover problem. The input to the greedy algorithm is the set of paths with the set of fibers used by each path and the set of all fibers. The greedy algorithm is an iterative algorithm that works as follows. In the first iteration, it selects a path using the minimum number of fibers, and updates the set of fibers not survived by the selected path. This greedy path selection is repeated until the selected path set survives all of the fibers. Following the proof of Lemma~\ref{K-OPT}, it can be shown that the greedy algorithm also finds a survivable path set with size at most $K+1$.

As discussed in Section \ref{MSP_general}, the greedy algorithm generally gives an $O(\log m)$ approximation to the minimum survivable path set. However, under the assumption of restricted path length, it provides a better approximation as stated in Theorem~\ref{K-greedy-app}.

\begin{theorem}\label{K-greedy-app}
The greedy algorithm provides an $O(\log K)$ approximation in polynomial time for the path length restricted version of MSP.
\end{theorem}
\begin{proof}
Let $\OPT$ be the size of minimum survivable path set. Let $n_{i}$ be the number of fibers that are not survived after the $i^{th}$ iteration of the greedy algorithm. Clearly, we have $n_{1} \leq K$. Now, note that there is a path that survives at least $\frac{n_{1}}{\OPT}$ of the remaining $n_{1}$ fibers, because otherwise the size of the optimal path set would be larger than $\OPT$. Hence, in the second iteration, the greedy algorithm would select a path that survives at least $\frac{n_{1}}{\OPT}$ of fibers. Thus,
\begin{equation}
n_{2} \leq n_{1}-\frac{n_{1}}{\OPT} \leq K(1-\frac{1}{\OPT}).
\end{equation}
Similarly, 
\begin{equation}
n_{3} \leq n_{2} - \frac{n_{2}}{\OPT} \leq K(1-\frac{1}{\OPT})^{2},
\end{equation}
and in general,
\begin{equation}
n_{i} \leq K(1-\frac{1}{\OPT})^{i}.
\end{equation}

The greedy algorithm will terminate when $n_{t} < 1$, and this condition is satisfied when 
\begin{equation}\label{greedy_inequality}
K(1-\frac{1}{\OPT})^{t} < 1,
\end{equation}
where $t$ is the total number of iterations. Since $1-x<e^{-x}$ for $x>0$, inequality (\ref{greedy_inequality}) is satisfied when 
\begin{equation}
Ke^{-\frac{t}{\OPT}}\leq 1 \Leftrightarrow t \leq \OPT\times \log K.
\end{equation}
Therefore, the greedy algorithm provides an $O(\log K)$ approximation.

To prove the polynomial time complexity,  note that in each iteration of the greedy algorithm, the best path can be found in $O(m^K)$ by searching through all the paths (see the proof of Theorem~\ref{Ksearch}). Furthermore, as mentioned above, the greedy algorithm terminates in at most $K+1$ iterations. Therefore, the computational complexity of the greedy algorithm is $O(Km^{K})$.
\end{proof}

Although the greedy algorithm runs significantly faster than the exhaustive search algorithm, its running time can still be prohibitive for large $K$ and $m$. Hence, we develop a novel randomized algorithm which has a considerably better running time. This algorithm builds upon solutions to the closely related Set Cover and Hitting Set problems. In particular, the algorithm is based on $\varepsilon$-net, a concept in computational geometry, which provides an approximation algorithm for the Hitting Set problem.

\subsubsection{$\varepsilon$-net Algorithm}\label{epsilon-net_description_section}
Our $\varepsilon$-net algorithm is an iterative algorithm which selects each path with some probability. If all the fibers are survived by the selected path set in the first iteration, the algorithm terminates. Otherwise, it changes the probability of selecting each path and selects a new set of paths using the new probabilities, until all fibers are survived. 

Let $W_j$ be the weight of path $j$, initialized as $W_j=1$. Define the weight of each fiber $i$ to be the sum of the weights of paths surviving fiber $i$, i.e.,
\begin{equation}
%W(f_i)=\sum_{\forall \mbox{ path } j \mbox{ surviving fiber } i} W_j.
W(f_i)=\sum_{j:a_{ij}=1} W_j.
\end{equation}
\begin{definition}
A fiber is said to be \RISKY if
\begin{equation}
W(f_i) \geq \varepsilon \sum_{j=1}^{n} W_j\mbox{ for some } \varepsilon\in(0,1), 
\end{equation} 
where $n$ is the total number of paths.
\end{definition}

Note that when all the paths have the same weight of $1$, a fiber is \RISKY if it is survived by at least $\varepsilon \times n$ paths. Hence, if a fiber is \RISKY with large $\varepsilon$, then it is likely to be survived by randomly selected paths. This observation is exploited in our $\varepsilon$-net algorithm as discussed below.

By applying the randomized algorithm for the hitting set problem from~\cite{Varadarajan} and~\cite{Matousek}, we can obtain a path-selection algorithm for selecting a random subset of paths that will survive all of the \RISKY fibers, with high probability. In particular, the algorithm finds a set of paths via $s$ independent random draws (with replacement), such that in each draw, a path is selected from the entire path set according to the probability distribution $\mu(P_j)=\frac{W_j}{\sum_{j=1}^{n}W_j},\forall j$. 

Our $\varepsilon$-net algorithm iteratively applies this random path selection as follows. After each iteration, it checks the survivability of the selected path set. If not all fiber failures are survived, the algorithm doubles the weight of all paths that survive the failure of fibers in $\bar S$, where $\bar S$ is the set all the fibers that are not survived yet (so that such fibers are more likely to be survived by the new path set). The random path selection is repeated with the new probability distribution.

Let $\OPT$ be the optimal solution to the MSP problem. By applying the results in \cite{HittingSet,Bronniman}, the following theorem can be proved.
\begin{theorem}\label{eps-net-thm-Krestricted}
Assume $s = c\frac{\log K}{\varepsilon} \log \frac{\log K}{\varepsilon}$, where $c$ is a constant. The $\varepsilon$-net algorithm finds a set of survivable paths of size $O(\log K \log \OPT)\OPT$, with high probability.
\end{theorem}

This theorem together with Lemma \ref{K-OPT} implies that the $\varepsilon$-net algorithm finds a survivable path set of size $O(\log ^2 K)\OPT$. Moreover, it can be shown that the algorithm requires $O(K \log(\frac{m}{K}))$ iterations to achieve this performance bound. On the other hand, the path-selection algorithm needs to select $O(\frac{\log K}{\varepsilon} \log \frac{\log K}{\varepsilon})$ paths in each iteration. Therefore, the computational complexity of the $\varepsilon$-net algorithm is $O(K \log (K) \log (m) \log(\log (K)) )$. Table \ref{K_table} summarizes the performance of each algorithm under the path length restriction.
\begin{table}[ht]
\centering
\begin{tabular}{|c|c|c|c|}
\hline
Method & Approximation & Running Time & T\\
\hline
ExS & Exact Solution & $O(m^{K(K+1)})$ & D\\
\hline
Greedy & $O(\log K)$ & $O(K m^K)$ & D\\
\hline
$\varepsilon$-net & $O(\log K \log \OPT)$ & $O(K \log (K) \log (m) \log(\log (K)) )$ & P\\
\hline
\end{tabular}
\caption{Performance bounds under path length restricted version: ExS-Exhaustive Search, T-Type, D-Deterministic, P-Probabilistic}
\label{K_table}\vspace{-0.3cm}
\end{table}

\subsection{Wavelength Restricted version}\label{sub_WDM}
Another important practical constraint is that in WDM-based networks, the number of lightpaths using a fiber is limited to say $W$, which is the number of wavelengths supported over a fiber. In this section, we assume that a set of logically disjoint paths with their mapping on the physical topology is given, and the goal is to find a minimum survivable path set among those paths under the WDM restriction. Note that the set of logically disjoint paths can be abstract to a logical topology with two nodes and parallel links (e.g., the one in Fig. \ref{fig:logical_TOP_MFSP}). Clearly, in this setting, the WDM restriction implies that \emph{each fiber can be used by at most $W$ paths}. Using this property, it can be shown that the MSP problem under the WDM restriction can be solved in polynomial time. To prove this, we need the following lemma.

\begin{lemma}\label{WDM-OPT}
Under the wavelength restriction, the minimum number of survivable paths is at most $W+1$.
\end{lemma}

\begin{proof}
Suppose that the minimum survivable path set contains more than $W+1$ paths. This implies that there exists a fiber whose failure disconnects at least $W+1$ paths (so that more than $W+1$ paths are needed for survivability), which contradicts to the fact that under the WDM restriction, each fiber can be used by at most $W$ paths.
\end{proof}

\begin{theorem}\label{WDM-complexity}
Under the wavelength restriction, the MSP problem can be solved in polynomial time.

\begin{IEEEproof}
It can be shown that under the WDM restriction, the given path set contains $O(m)$ paths. By Lemma \ref{WDM-OPT}, we only need to enumerate path sets of size up to $W+1$ in order to find a minimum survivable path set. Clearly, this can be done in $O(m^{W+1})$ time. More details can be found in Appendix \ref{WDM_polynomial}.
\end{IEEEproof}
\end{theorem}

Although there exists a polynomial time optimal algorithm, it requires excessive computation for large values of $W$ and $m$. As in the case of restricted path length, we have developed approximation algorithms with better running time. Table \ref{WDM_table} shows the summary of our approximation algorithms under the wavelength restriction (See~\ref{epsnet_WDM} for details).

\begin{table}[ht]\label{WDM_table}
\centering
\begin{tabular}{|c|c|c|c|}
\hline
Method & Approximation & Running Time & T\\
\hline
ExS & Exact Solution & $O(W^{W+1}m^{W+1})$ & D\\
\hline
Greedy & $O(\log m)$ & $O(W^2 m)$ & D\\
\hline
$\varepsilon$-net & $O(\log W \log \OPT)$ & $O(W \log (W)  \log (m) \log(\log (W)))$ & P\\
\hline
\end{tabular}
\caption{Approximation bounds under wavelength restricted version: ExS-Exhaustive Search, T-Type, D-Deterministic, P-Probabilistic}
\label{WDM_table}\vspace{-0.4cm}
\end{table}

\section{Minimum number of physical fibers in Survivable Paths (MFSP)}
Our focus so far has been on providing protection using the minimum number of paths. In the this section, our goal is to find a survivable path set that uses the minimum number of fibers. This problem seems to have a direct connection to the minimum cost survivable path set problem where the cost of a path is the number of fibers used by that path. However, this is not true owing to the fact that costs of paths are not additive, i.e., a fiber that is used by multiple paths only adds one unit of cost. In order to make this point clear, consider Fig. \ref{additive cost}. A minimum cost survivable path set problem will find paths 1 and 2 as the set of survivable paths with total cost 7, while the MFSP problem will find paths 2, 3 and 4 as the optimal survivable path which has the total cost 6. In the next section we will develop ILP formulations, and analyze the complexity of MFSP.

\begin{figure}[ht]
\centering
\subfigure[Logical Topology]
{\label{fig:logical_TOP_MFSP}\includegraphics[scale=0.38]{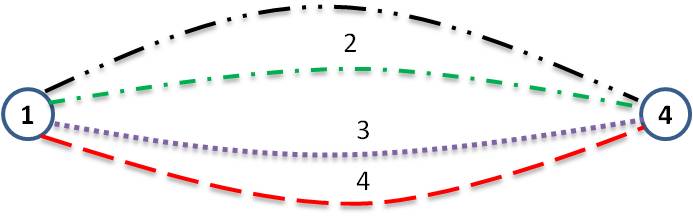}}                
\subfigure[Routing]
{\label{fig:ROUTING_MFSP}\includegraphics[scale=0.38]{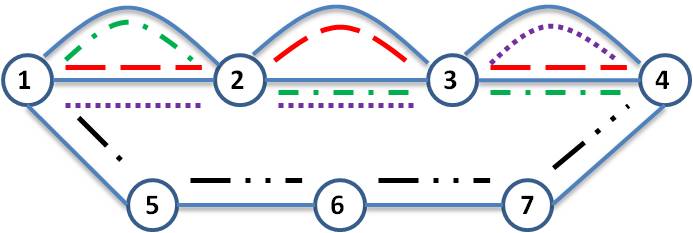}}
\caption{Routing in Multilayer Network}
\label{additive cost}
\end{figure}

\subsection{ILP Formulation and Complexity}
We start with an ILP formulation of the problem. Similar to the MSP problem, the MFSP problem can be formulated in several different ways, but here we only present the path-based formulation which will be used for developing heuristics and approximation algorithms. Given the set of paths and associated fibers, for each path $j$, assign a binary variable $P_j$ which takes the value 1 if path $j$ is selected and 0 otherwise. Similarly, for each fiber $i$, assign a binary variable $f_i$ which takes the value 1 if fiber $i$ is selected and 0 otherwise. The matrix $A$ and vector $e$ are defined in the same way as in the MSP formulation (\ref{eqn:msp-obj})-(\ref{eqn:msp-bin}).
\begin{align}
MFSP: \quad\quad \mbox{minimize}  \quad& \sum_{i=1}^m f_{i}\\
\quad\quad\quad \mbox{subject to} \quad& A \times P \geq e \label{MFSP_survive_const}\\
\quad\quad\quad\quad&f_{i} \geq P_{j} \quad\quad  \forall f_{i} \in P_{j} \label{MFSP_fiber_path_const} \\
\quad\quad\quad\quad& P_{j} \in \{0,1\} \quad \forall P_j
\end{align}
In the above, the objective function is the number of fibers used by the selected paths. Again, the constraints in (\ref{MFSP_survive_const}) require the selected path to be survivable. The constraints in (\ref{MFSP_fiber_path_const}) relate the selected paths and fibers, such that a fiber is selected if at least one of the paths using the fiber is selected. Clearly, the optimal solution to the above optimization problem gives a set of survivable paths that use the minimum number of fibers. This MFSP problem can be shown to be NP-hard.

\begin{theorem}
Computing the set of survivable paths using the minimum number of physical fibers is NP-hard.
\end{theorem}

\begin{proof}
We provide a mapping from the Minimum 3-Set Cover problem, which is a special version of the Set Cover problem where each set has exactly 3 elements, to the MFSP problem. The Minimum 3-Set Cover problem is NP-hard, and holds all the inapproximability properties of the Minimum Set Cover problem.% Our mapping uses the physical topology shown in Fig. \ref{complexity-MFSP-fig}. 
\begin{figure}[ht]
	\begin{center}
	\includegraphics[scale=0.35]{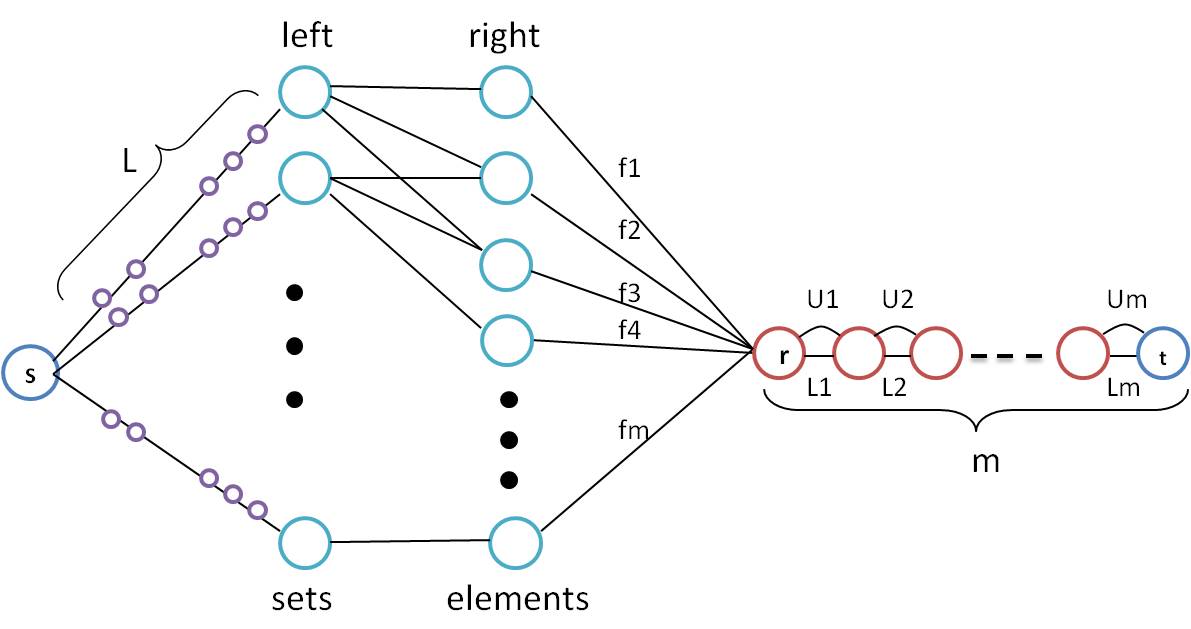}
	\end{center}
	\caption{Physical Topology}
	\label{complexity-MFSP-fig}
\end{figure}

Consider an instance of the Minimum Set Cover problem with the ground set $E$ and a family of subsets $F$. Suppose that each subset in $F$ contains only 3 elements. To show a mapping, we construct a physical topology as shown in Fig.~\ref{complexity-MFSP-fig}, such that each node on the left corresponds to a subset in $F=\{C_{1},...,C_{|F|}\}$ and the nodes on the right are the elements of $E=\{e_{1},...,e_{m}\}$. Node $j$ on the left is connected to node $i$ on the right if and only if $e_{i} \in C_{j}$. Note that a node on the left is connected to only three nodes on the right (i.e., each set contains only three elements).

We can construct a logical topology and its lightpath routing over the physical topology; such that for protection, we need to have $m$ paths from $s$ to $t$ that pass through all the nodes on the right. Moreover, since each path between $s$ and the nodes on the left uses a large number of fibers, we should select a survivable path set that uses the minimum number of nodes on the left. Consequently, the minimum fiber survivable path set for the aforementioned layered network gives a minimum set cover for the given instance of $E$ and $F$, which shows the NP-hardness of the MFSP problem. For the complete proof, see Appendix \ref{MFSP-complexity-prf-app}.
\end{proof}

Since the MFSP problem is a reduction from the minimum 3-set cover problem, it is unlikely that the MFSP problem has an efficient optimal algorithm. For this reason, we develop new heuristics and approximation algorithms. In particular, as in the previous section, we focus on the practical scenario where the number of paths on a fiber is at most $W$, i.e., the wavelength restricted setting. We first present a greedy algorithm, and then a randomized rounding algorithm based on the path-based formulation for MFSP.

\subsection{Additive Cost Greedy Algorithm (ACG)}\label{Greedy-in-MFSP}
Recall that the goal is to find a survivable path set that uses the minimum number of fibers. Hence, it is desired to select a path that uses a small number of fibers while surviving many new fibers (i.e., fibers not survived by already selected paths) as possible. Note that this is clearly different from the MSP problem where the number of fibers does not matter. The Additive Cost Greedy algorithm requires the set of paths and associated fibers as input. We define a new cost metric in order to take into account the two factors simultaneously. Let $C_j$ be the number of fibers used by path $j$. The ``amortized cost" $AC_j$ of path $j$, which is updated for every iteration, is defined as follows:
\begin{align}
AC_{j}=\frac{C_{j}}{\# \mbox{newly survived fibers by } P_{j}}, \nonumber
\end{align}
where the denominator is the number of fibers survived by path $j$ and not survived by the previously selected paths. Our greedy algorithm selects a path with minimum amortized cost, updates the amortized costs of the remaining paths, and continue until all the fibers are survived. This greedy algorithm, which we call the Additive Cost Greedy algorithm, gives an approximate solution.

\begin{theorem}\label{cost_greedy}
The Additive Cost Greedy algorithm provides an $O(W\log m)$ approximation to the MFSP problem.
\end{theorem}

Note that the number $C_j$ in the additive cost of path $j$ does not change over iterations. That is, the additive cost implicitly assumes that selecting path $j$ will add $C_j$ fibers to the total cost, while only the number of new fibers is added to the total cost. Therefore, one can better take into account the actual change to the cost by updating $C_j$ as the number of fibers that are used by path $j$ and not used by the previously selected paths. In Section \ref{simulation_sec}, we will show that this Non-additive Cost Greedy (NACG) algorithm works better than the ACG, by finding survivable path sets with fewer fibers.

\subsection{Randomized Rounding Algorithm}\label{RND_SEC}
Randomized rounding is a widely used technique to solve difficult integer optimization problems. In general, randomized rounding scheme solves the Linear Program (LP) relaxation of the original ILP formulation, and rounds the solution randomly. In our case, the LP relaxation of the MFSP problem is given as
\begin{align}
\mbox{LP relaxation:}\quad \mbox{minimize} \quad&\sum_{i=1}^m f_{i} &\\
\quad\quad\quad\mbox{subject to} \quad&A \times P \geq e & \label{A_P_e}\\
\quad\quad\quad&f_{i} \geq P_{j} \quad& \forall f_{i} \in P_{j} \label{path-fiber} \\
\quad\quad\quad&0 \leq P_{j} \leq 1. & \label{relaxed_const} 
\end{align}
Let $P_j^*$ and $f_i^*$ be the optimal values of path $j$ and fiber $i$. Note that the above path-based LP uses the set of paths and associated fibers as input. Our randomized rounding algorithm to solve the MFSP problem works as follows:
\begin{enumerate}
\item{Initialize $S=\emptyset$. Solve the relaxed problem.}
\item{Select each path $j$ with probability $P_j^*$, and add it to $S$ if selected.}
\item{Repeat step 2 for $T$ times}.
\end{enumerate}

Since paths are selected randomly, some fibers may not be survived in one iteration. Clearly, as the number of iterations $T$ increases, the probability of surviving all of the fibers increases. On the other hand, it may increase the number of selected paths and thus fibers. Therefore, the parameter $T$ determines the survivability probability and the approximation quality of the solution. The following theorem characterizes this relationship.

\begin{theorem}\label{randomized-rounding-thm}
With $T=O(\log \frac{m}{1-q})$ iterations, the randomized rounding algorithm gives an $O(W\log \frac{m}{1-q})$ approximation with probability at least $q$.
\end{theorem}

\begin{proof}
We first find an upper bound on the expected number of fibers selected in each iteration (which gives the approximation quality of the solution), and then, the probability of survivability is derived.

\subsubsection{Expected Number of Selected Fibers}
Note that fiber $i$ is selected if any of the paths using the fiber is added to the path set $S$. Moreover, in each iteration, each path $j$ is added with probability $P_j^*$. To count the number of selected fibers, define a random variable $F_i$ for each fiber $i$ such that $F_i=1$ if fiber $i$ is selected and $0$ otherwise. The expected number of fibers selected in each iteration can be written as
\begin{align}
E[\sum_{i=1}^{m} F_{i}]% \quad \quad \nonumber \\ 
=\sum_{i=1}^{m} \pr(F_{i}=1) = \sum_{i=1}^{m} (1-\pr(F_{i}=0)) \label{cost_iter}
\end{align} 

Therefore, we need to compute $\pr(F_{i}=0)$, which is the probability of a fiber not being selected. Note again that a fiber is not selected if none of the paths using the fiber are selected. It follows that
\begin{align}
\pr(F_{i}=0)% = \pr(\mbox{none of paths using} \  f_{i} \ \mbox{are selected}) \nonumber\\
&=\prod_{j:a_{ij}=0} (1-P_{j}^*)\\
&\geq \prod_{j:a_{ij}=0} (1-f_{i}^*) \quad\mbox{(by constraint (\ref{path-fiber}))} \label{bound_fibers}
\end{align}
where the equality is due to the independence of path selections. Let $w_{i}$ be the number of paths that use fiber $f_{i}$, i.e., $w_i=|\{j:a_{ij}=0\}|$. Then, we can obtain
\begin{equation}
\pr(F_{i}=0) \geq \prod_{j:a_{ij}=0} (1-f_{i}^*) = (1-f_{i}^*)^{w_{i}}, \label{product_fibers}
\end{equation}
\begin{align}
&\pr(F_{i}=1) \leq 1-(1-f_{i}^*)^{w_{i}}.
\end{align}

Finally, by using the fact that $1-(1-x)^n\leq nx$, the probability of selecting a fiber can be upper-bounded as 
\begin{align}
\pr(F_{i}=1) \leq w_{i}f_{i}^*.    \label{upperbound_select_fiber}
\end{align}
Combining (\ref{cost_iter}) and (\ref{upperbound_select_fiber}) yields the following bound on the expected number of fibers selected in each iteration:
\begin{align}
E[\sum\limits_{i=1}^{m} F_{i}]&\leq \sum\limits_{i=1}^{m} w_{i}f_{i}^*\nonumber\\
&\leq W \times \sum\limits_{i=1}^{m}f_{i}^* \quad\mbox{(by wavelength restriction)}\nonumber\\
&= W \times LP^*, \label{final_cost}
\end{align}
where $LP^*$ is the optimal value of the LP relaxation.

\subsubsection{Probability of Survivability}
Next, we derive an upper bound on the probability that the selected path set is \textit{not} survivable, by applying the idea of the feasibility argument in~\cite{Raghavan}. First, for each fiber $i$, the probability that the selected path set cannot survive the failure of fiber $i$ can be written as follows:
\begin{align}
&\pr(\mbox{fiber $i$ not survived in one iteration}) \\
&= \pr(\mbox{none of paths surviving fiber $i$ are picked}) \\
%&= \prod_{j:a_{ij}=1}\pr(P_{j} \  \mbox{not picked})\\
&= \prod_{j:a_{ij}=1} (1-P_{j}^*) \\
&\leq \prod_{j:a_{ij}=1} e^{-P_{j}^*} \quad\quad\quad\quad\quad \mbox{using} (1-x \leq e^{-x}) \\
&\leq e^{-\sum_{j:a_{ij}=1}P_{j}^*} \leq \frac{1}{e}. \quad\quad \mbox{(using constraint \ref{A_P_e}) }
\end{align}
Since the randomized rounding runs for $T$ iterations with $T=\log \frac{m}{1-q}$, we can obtain
\begin{equation}
\pr(f_{i} \ \mbox{not covered in all iterations}) \leq \frac{1}{e^{\log \frac{m}{1-q}}} = \frac{1-q}{m}.
\end{equation}
Thus, by the union bound,
\begin{equation}
\pr( \mbox{there exist an unsurvived fiber}) \leq m \times \frac{1-q}{m}= 1-q.
\end{equation}

\subsubsection{Approximation Result}

By (\ref{final_cost}), the total expected number of fibers after $T$ iterations is bounded as
\begin{equation}
E[\mbox{Total \# fibers}] \leq W \log{\frac{m}{1-q}} LP^*
\end{equation}
Since the solution is in integer form, it is an upperbound for the ILP solution. Thus, with probability at least $q$,
\begin{align}
\frac{E[\mbox{Total \# fibers}]}{ W \log{\frac{m}{1-q}} } \leq ILP \leq E[\mbox{Total \# fibers}].
\end{align}

\end{proof}

\subsection{Random-Sweep Greedy (RSG)} \label{RSGsec}

Next, we present a new Greedy algorithm for the MFSP problem, which is called the Random-Sweep greedy. Unlike the Greedy algorithm discussed in Section \ref{MSP_SEC}, the RSG removes a path (from the selected path set) which survives the fibers covered by other paths; so that the size of the path set can be further reduced while maintaining the survivability. Although we could not quantify the performance of this algorithm, it performs near optimally in some scenarios as will be shown in Section \ref{simulation_sec}.

The RSG algorithm also requires the knowledge of the set of paths and associated fibers. Let $S_j$ be the set of fibers that are survived by path $j$. Moreover, let the cost $C_j$ of path $j$ in each iteration be the number of fibers that are used by path $j$, and not used by the previously selected paths. Using the cost function $C_j$, define the amortized cost $AC_j$ as the ratio of $C_j$ to the number of newly survived fibers by path $j$. The first two iterations of RSG are the same as the Non-Additive Cost greedy algorithm. That is, in each iteration, it selects a path with minimum amortized cost. If the first two paths survive all of the fibers, the algorithm terminates. Otherwise, it continues as follows.

Suppose the RSG algorithm is in the $i^{th}$ iteration. First, find a path, say $i$, with minimum amortized cost among the remaining paths. Then, pick a path, say $j$, randomly from the previously selected paths and find $S_j \cup S_i$, which is the set of fibers that are survived by either path $i$ or path $j$. If there exists a path $k$ among the previously selected paths such that $S_k \subset S_i \cup \S_j$, remove path $k$ from the selected paths. Note that removing such a path does not affect the survivability of the selected path set, i.e., the same set of fibers are still survived after the removal. More importantly, we can possibly decrease the number of fibers used by the selected paths.

Table \ref{MFSP_table} shows the summary of our algorithms for the MFSP problem. Note that we have also developed an $\varepsilon$-net algorithm and its details can be found in~\ref{epsnet_MFSP}.

\begin{table}\label{MFSP_table}
\centering
\begin{tabular}{|c|c|c|c|}
\hline
Method & Approximation & Running Time & T\\
\hline
ExS & Exact & $O(W^{W+1} m^{W+1})$ & D\\
\hline
ACG& $W \log m$ & $O(W^2 m)$ & D\\
\hline
$\varepsilon$-net & $W \log W \log \OPT$ & $O(W \log (W) \log (m) \log(\log (W)) )$ & P\\
\hline
RR & $W \log{\frac{m}{1-p}}$ & $\log{\frac{m}{1-p}}$ & P\\
\hline
RSG & nearly Opt. & $O(W^2 m)$ & D\\
\hline
\end{tabular}
\caption{Approximation bounds under wavelength restricted version of MFSP: ExS-Exhaustive Search, RSG-Random Sweep Greedy, ACG-Additive Cost Greedy, RR-Randomized Rounding, T-Type, D-Deterministic, P-Probabilistic}
\label{MFSP_table}
\end{table}

\section{Simulation Results}\label{simulation_sec}
We compare the performance of our algorithms using both large-scale random network topologies, as well as the US backbone network topology. In particular, we compare the following algorithms:
\begin{itemize}
\item ILP-based optimal algorithm computed by CPLEX; denoted by ILP
\item Simple Greedy algorithm from Section \ref{MSPG-SEC}; denoted by MSPG
\item Additive Cost Greedy algorithm from Section \ref{Greedy-in-MFSP}; denoted by ACG
\item Non-additive Cost Greedy algorithm from Section \ref{Greedy-in-MFSP}; denoted by NACG
\item Random-Sweep Greedy algorithm from Section \ref{RSGsec}; denoted by RSG
\item Randomized rounding algorithm from Section \ref{RND_SEC}; denoted by RR
\item $\varepsilon$-net algorithm from Section \ref{epsilon-net_description_section}; denoted by EPS
\end{itemize}

\subsection{Performance in Large-scale Random Topologies}
We first consider a random layered network where the logical topology consists of 50 paths between nodes $s$ and $t$. This layer is mapped onto the physical topology containing 100 fibers, using the mapping structure shown in~\cite{Kayi}. In the wavelength restricted version of the problem, at most $W$ paths can be assigned to each fiber. For each value of $W$, we generate 1000 random topologies each with 50 paths that are randomly routed on the physical topology. We then apply our algorithms to each network in order to find a survivable path set using the minimum number of fibers (i.e., to solve the MFSP problem). Note that for Randomized Rounding the performance depends on the survivability guarantee of the algorithm, which is 99.9\% for the results shown below.

\begin{figure}[ht]
\centering
\includegraphics[scale=0.055]{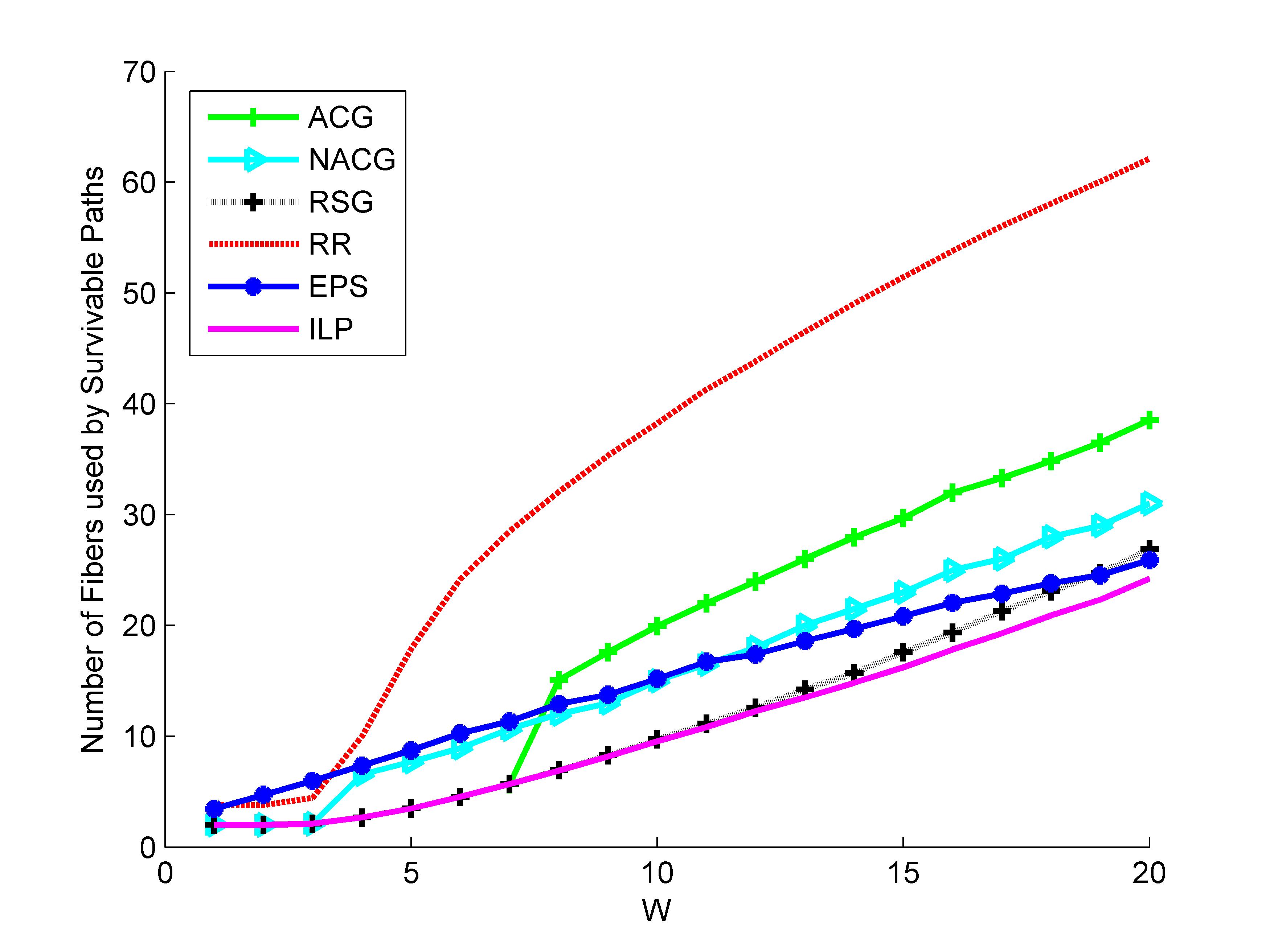}\vspace{-0.3cm}
\caption{Comparison of algorithms for MFSP problem: Approximation quality in random networks}
\label{Compare-algs}\vspace{-0.2cm}
\end{figure}

Fig. \ref{Compare-algs} compares the average number of fibers in the survivable path set found by each algorithm. It can be seen that as the value of $W$ increases, the number of used fibers increases. This is due to the fact that when $W$ is large, more logical paths can share a fiber, and therefore more logical paths are needed since a single physical link failure can lead to a large number of logical path failures. Note that the Random-Sweep Greedy (RSG) algorithm is nearly optimal, and the performance of $\varepsilon$-net algorithm is better than RSG for large values of $W$. Fig. \ref{Compare-RunTime} compares the logarithm of the running time of the algorithms. It can be seen that the Randomized Rounding algorithm is the fastest, while the RSG algorithm which gives the closest to optimal solution, and the $\varepsilon$-net algorithm which performs nearly optimally for networks with large values of $W$, have larger running times. Note also that the running times are nearly independent of $W$ for all of the proposed algorithms. In contrast, obtaining the exact optimal solution using the ILP formulation becomes quickly impractical as $W$ increases.

\begin{figure}[ht]
\centering
\includegraphics[scale=0.055]{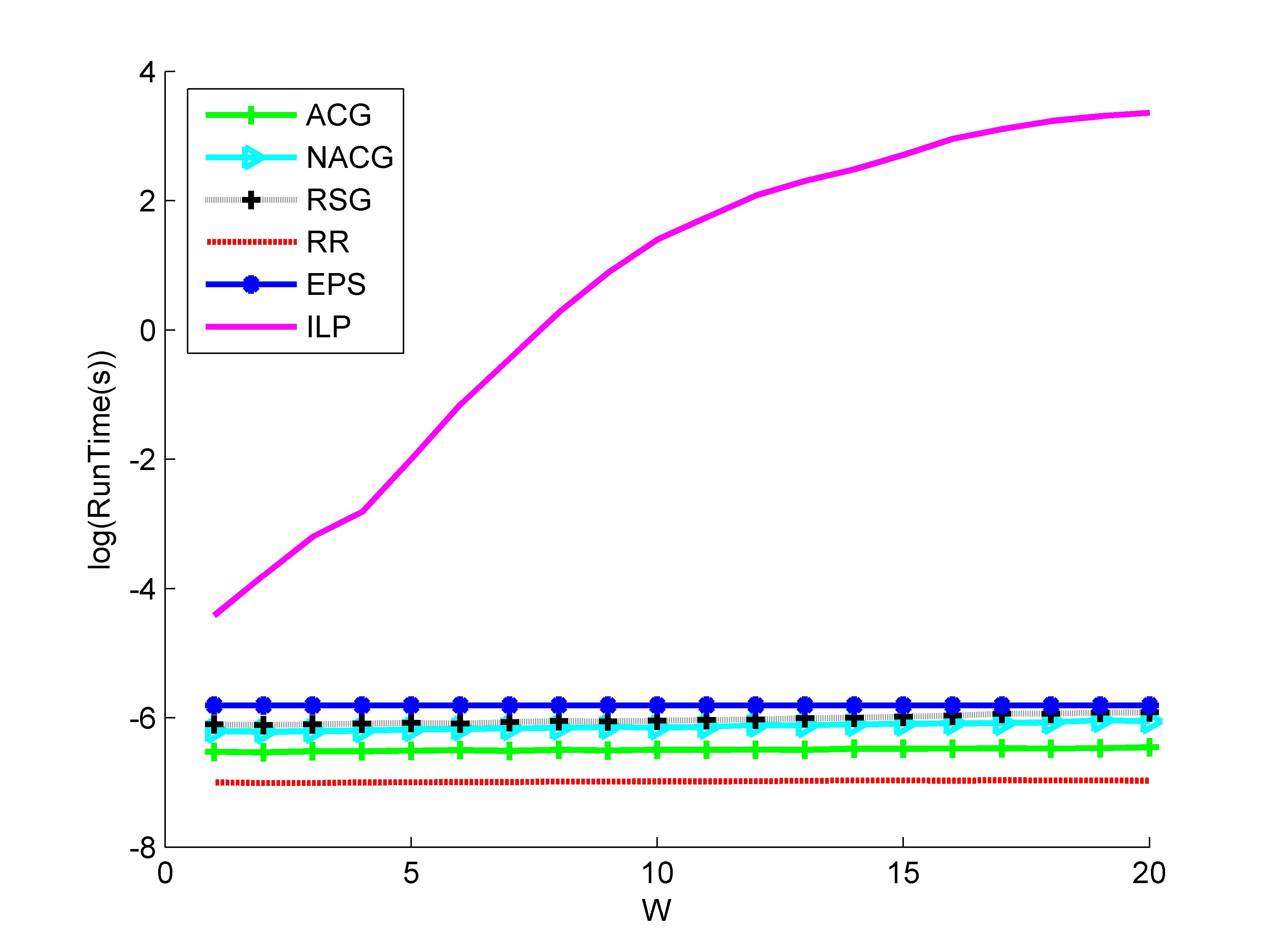}\vspace{-0.3cm}
\caption{Run Time Comparison of Different heuristics with respect to Optimal}
\label{Compare-RunTime}\vspace{-0.2cm}
\end{figure}

Next, we consider larger networks where there are 1000 fibers in the physical topology and 500 paths in the logical topology, with $W$ ranging from 1 to 40. Fig. \ref{MFSP_largeW} shows the performance of the various algorithms as a function of $W$. The performance of the ILP-based algorithm is omitted since CPLEX often fails to find a solution within a reasonable amount of time. Again we see that the RSG algorithm considerably outperforms the rest of algorithms.

\begin{figure}[ht]
\centering
\includegraphics[scale=0.055]{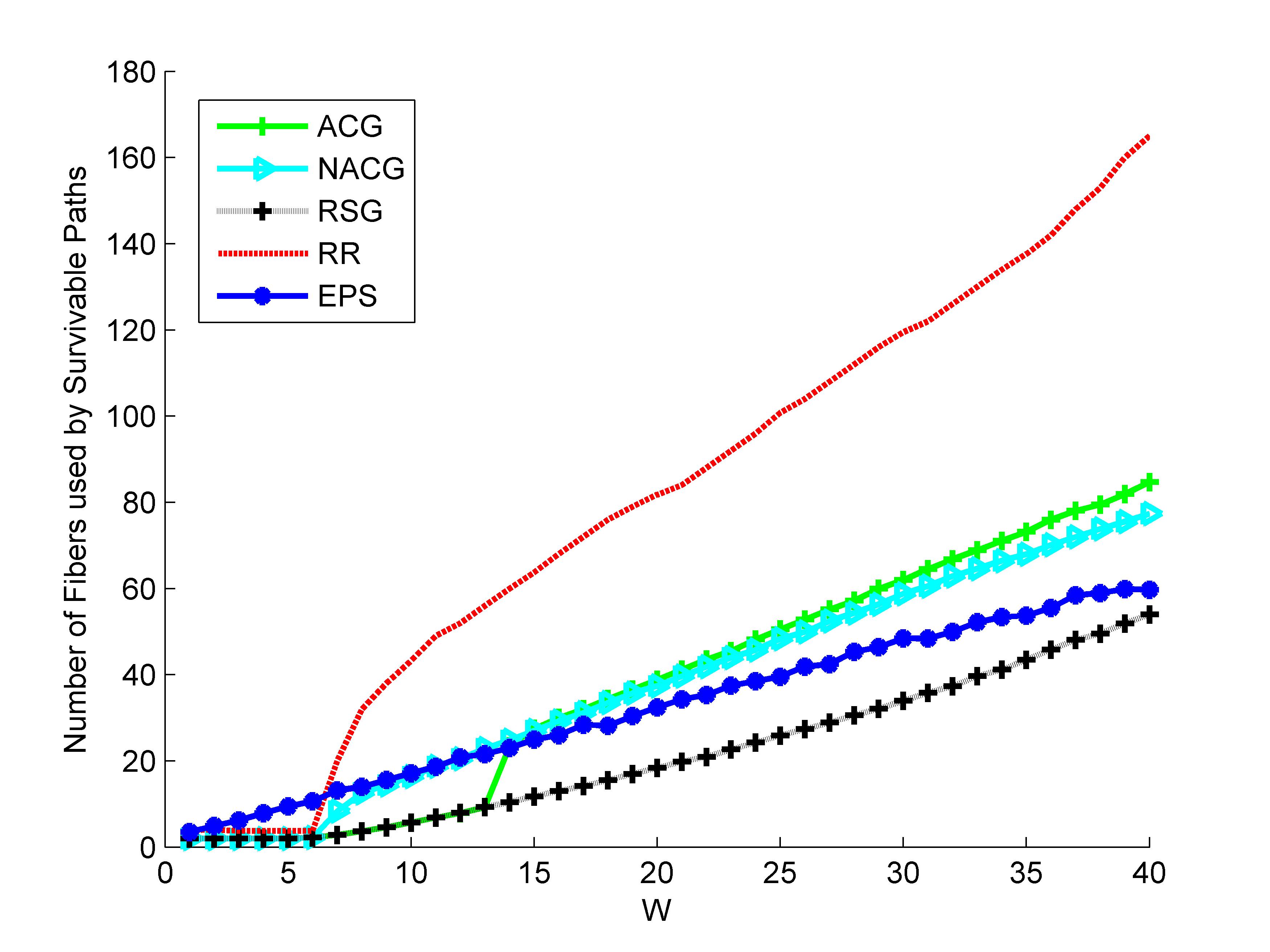}\vspace{-0.4cm}
\caption{Comparison of Approximation Algorithms in Large Networks}
\label{MFSP_largeW}\vspace{-0.2cm}
\end{figure}

\subsection{Performance in Real Networks}

\begin{figure}[ht]
\centering
\includegraphics[scale=0.3]{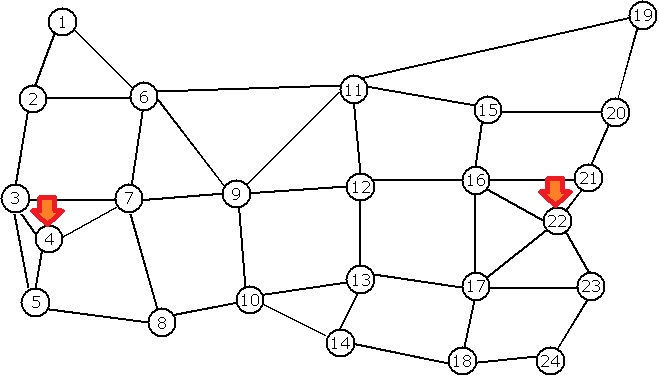}
\caption{Physical Topology}
\label{IP}
\end{figure}

Next, we examine the performance of the approximation algorithms over the US backbone topology shown in Fig. \ref{IP}, with the objective of finding a minimum survivable path set between nodes 4 and 22~\cite{IPbackbone}. For the logical topology, we generated random graphs with eight nodes (including 4 and 22) each of degree 4. We use shortest path lightpath routing for the logical links.

Table \ref{realnetwork} shows the average number of paths and average running time of each algorithm. It can be seen that the RSG and randomized rounding algorithms are nearly optimal, and furthermore, the randomized rounding gives a solution almost instantly. We also note that the survivability guarantee of the Randomized Rounding algorithm is 99\% for the results shown in the table.

\begin{table} 
\centering
\begin{tabular}{|c|c|c|}
\hline
Method & Number of Paths & Running Time (ms)\\
\hline
ILP & 2.0069 & 7.2133\\
\hline
RSG & 2.0160 & 2.0167\\
\hline
RR & 2.0482 & 0.0272\\
\hline
MSPG & 2.2241 & 0.1911\\
\hline
EPS & 2.551 & 1.6000\\
\hline       
\end{tabular}
\caption{Comparison of Algorithms for MSP in Real Networks}
\label{realnetwork}\vspace{-0.3cm}
\end{table}

\section{Conclusion}
We considered the problem of finding survivable paths in layered networks. The traditional disjoint paths approach for protection cannot be directly applied to layered networks, since physically disjoint paths may not always exist in such networks. To address this issue, we introduced the new notion of \emph{survivable path set}. We showed that in general the problem of finding the minimum size survivable path set (MSP) and the problem of finding the minimum fiber survivable path set (MFSP) are NP-hard and inapproximable. However, under practical constraints, we are able to develop both optimal and approximation algorithms for the MSP and MFSP problems. An important future direction is to develop backup routing schemes based on survivable path sets.

\bibliographystyle{IEEEtran}
\bibliography{reference}

\clearpage

\appendices

\section{Proof of Complexity of MSP}\label{MSP_proof}
The Minimum survivable paths problem can be reduced from the NP-hard Minimum Set Cover problem.
Given an instance of Minimum Set Cover Problem with ground set $E$ and family of subsets $R$, we construct a physical topology $E=\{f_1,\cdots,f_m\}$ containing all $m$ fibers and a logical topology $R=\{P_1,\cdots,P_n\}$, where each $P_j$ corresponds to the set of fibers that survive in the failure of path $j$, i.e. all fibers that are \textit{not} used by path $j$. It follows that the minimum number of logical paths that survives all the physical fibers is equal to the size of a minimum set cover.
As the last step of proof, we need to show we can construct a physical topology with the given routing.
Given the set of paths and the fibers used by each path (complement of fibers survived by each path), we can use the physical topology in~\cite{Kayi}.
The inapproximability result follows immediately from the inapproximabilities of the Minimum Set Cover problem.\\

%%%%%%%%%%%%%%%%%%%%%%%%%%%%%%%%%%%%%%%%%%%%%%%%%%%%%%%%%%%%%%%%
%%%%%%%%%%%%%%%%%%%%%%%%%%%%%%%%%%%%%%%%%%%%%%%%%%%%%%%%%%%%%%%%

\section{Proof of Theorem \ref{eps-net-thm-Krestricted}}\label{VC-dim}

In the procedure of $\varepsilon$-net algorithm, the ``path-selection" algorithm will be applied iteratively, and checks the survivability of the selected path set after each iteration. If not all fibers are survived, the algorithm doubles the weight of all paths that survive the failure of fibers in $\bar S$, where $\bar S$ is the set all the fibers that are not survived yet, and repeat the random path selection.

Let $\OPT$ be the optimal solution of MSP. Based on the results in \cite{HittingSet,Bronniman}, it can be shown that if in each iteration the selected subset of paths survive a ``good" subset of fibers, in $O(\OPT \log(\frac{m}{\OPT}))$ iterations, the algorithm will return a set of survivable paths, with high probability. A subset is ``good" if it is an $\varepsilon$-net.

\begin{definition}
Consider a set system $F = (X,R)$, where $X$ is the set of elements and $R$ is the set of subsets of $X$. A set $H \subset X$ is an ``\epsnet" of $F$ if $S \cap H \neq \emptyset $, for every subset $S \in R$ for which $|S| \geq \varepsilon |X|$.
\end{definition}

Lemma \ref{Finder-THM} claims that it is guaranteed that in each iteration the selected paths survive a ``good" subset of fibers

\begin{lemma}\label{Finder-THM}
For all $\varepsilon \in (0,\frac{1}{2})$, if $s = c\frac{\log K}{\varepsilon} \log \frac{\log K}{\varepsilon}$, where $c$ is a constant, the path-selection algorithm selects a subset of paths that survives all of the \RISKY fibers with high probability.
\end{lemma}

For the proof of Lemma \ref{Finder-THM}, we use the new techniques found by Haussler and Welzl in~\cite{Haussler}. Theorem \ref{thm-epsilon-net} is an improvement on their work~\cite{Matousek,HittingSet}.

Before presenting Theorem \ref{thm-epsilon-net}, we need to define VC-dimension.

\begin{definition}
Let $R$ be a set system on a set $X$. Let us say that a subset $A \subset X$ is shattered by $R$ if each of the subsets of $A$ can be obtained as the intersection of some $S \in R$ with $A$, i.e. if $R|_{A}=2^{A}$.

Define the VC-dimension of $R$, denoted by $dim(R)$, as the supremum of the sizes of all finite shattered subsets of $X$. If arbitrarily large subsets can be shattered, the VC-dimension is $\infty$.
\end{definition}

\begin{theorem}\label{thm-epsilon-net}
Let $F = (X,R)$ denote a set system with weights $w(u)$. For every $\varepsilon \in (0,\frac{1}{2})$, a random sample of $X$ according to the probability distribution $w(u)=w(X)$ is likely to to be an \epsnet  with respect to $w(u)$, if the sample contains $O(\frac{d}{\varepsilon} \log (\frac{d}{\varepsilon}))$ elements, where $d$ is the VC-dimension of the set system.
\end{theorem}

To prove Lemma \ref{Finder-THM}, it is enough to show that $O(\frac{\log K}{\varepsilon} \log (\frac{\log K}{\varepsilon}))$ paths are needed to cover all \RISKY fibers.

Let $X=\{P_1,\cdots,P_n\}$ be the set of paths in our problem, and $R=\{f_1,\cdots,f_m\}$ be the set of subsets of $X$, where each fiber-set $f_i$ corresponds to the set of paths that survives the failure of fiber $i$. Therefore, an \epsnet will cover all \RISKY fibers. 

In this setting, a subset $A$ of paths is shatterd if intersection of $A$ with every fiber-set produces all the subsets of $A$. For instance, $A=\{P_{1},P_{2}, P_{3}\}$ is shattered by $R$ if there exist 8 fibers $f_{1},f_{2},f_{3},f_{4},f_{5},f_{6},f_{7},f_{8}$ such that $\{A\cap f_{1}=P_{1}, A\cap f_{2}=P_{2}, A\cap f_{3}=P_{3}, A \cap f_{4} = \{P_{1},P_{2}\}, A \cap f_{5} = \{P_{1},P_{3}\}, A \cap f_{6} = \{P_{2},P_{3}\}, A \cap f_{7} = A, A \cap f_{8} = \emptyset\}$.

\begin{lemma}\label{VC-dim-K}
In path length restricted version of MSP, VC-dimension $d$ is less than $\log K$.
\end{lemma}
\begin{proof}
Suppose VC-dimension is $d$. Then, by definition, there exist a subset of paths $A$ of size $d$ which intersection of $A$ with all fiber-sets generates all subsets of $A$. In particular, for every $P_{j} \in A$, half of the subsets created by $A \cap R$ should contain $P_{j}$ and the other half should not contain it. 

Under the path length restricted assumption, each path uses at most $K$ fibers, and survives at least $m-K$ remaining fibers. Therefore, at least $m-K$ fibers contain a particular path $j$. Thus, 
\begin{equation}\label{d1K}
2^{d-1} \leq m-K \mbox{and} d \leq 1+\log (m-K).
\end{equation}

On the other hand, for each $P_{j}$ at most $K$ fibers do not contain it. Hence, 
\begin{equation}\label{d2K}
2^{d-1} \leq K \mbox{and} d \leq 1+\log K.
\end{equation}

By combining both equations \ref{d1K} and \ref{d2K}, we will have the following result:
\begin{equation}
d \leq 1+\log K.
\end{equation}
\end{proof}

%%%%%%%%%%%%%%%%%%%%%%%%%%%%%%%%%%%%%%%%%%%%%%%%%%%%%%%%%%%%%%%%
%%%%%%%%%%%%%%%%%%%%%%%%%%%%%%%%%%%%%%%%%%%%%%%%%%%%%%%%%%%%%%%%

\section{proof of WDM polynomialty}\label{WDM_polynomial}
Under the assumption in Section \ref{sub_WDM}, we are given a set of logically disjoint paths that satisfy the WDM restriction. As explained, this setting implies that each fiber can be used by at most $W$ paths. Therefore, we have the following Lemma:

\begin{lemma}\label{WDM-paths}
Under the WDM restriction, the number of given paths can be at most $W\cdot m$.
\end{lemma}
\begin{proof}
By the Max-Flow Min-Cut Theorem, in the physical topology the number of disjoint paths between nodes $s$ and $t$ is equal to the minimum $s-t$ cut ($MC$). On the other hand, since a fiber can be used by at most $W$ logical links, each physical path can carry at most $W$ logical links. Therefore, the maximum possible number of logical paths is $W\cdot MC \leq W \cdot m$.
\end{proof}

%%%%%%%%%%%%%%%%%%%%%%%%%%%%%%%%%%%%%%%%%%%%%%%%%%%%%%%%%%%%%%%%
%%%%%%%%%%%%%%%%%%%%%%%%%%%%%%%%%%%%%%%%%%%%%%%%%%%%%%%%%%%%%%%%
\section{epsilon net in WDM}\label{epsnet_WDM}

Using the same techniques discussed in the Section \ref{VC-dim}, we have the following Theorem:

\begin{theorem}\label{WDM-Finder-verifier}
The $\varepsilon$-net algorithm finds a set of survivable paths of size $O(\log W \log \OPT)\OPT$, with high probability and terminates in $O(\OPT \log(\frac{m}{\OPT}))$ iterations.
\end{theorem}

To prove Theorem \ref{WDM-Finder-verifier}, we need Lemma \ref{THM-FINDER-WDM}. Then, by an argument similar to the one to prove Thm \ref{eps-net-thm-Krestricted}, the $\varepsilon$-net algorithm will find a set of survivable paths with a $\log W \log \OPT$ approximation bound.

\begin{lemma}\label{THM-FINDER-WDM}
$\forall \varepsilon \in (0,1)$, if $s=c\frac{\log W}{\varepsilon} \log \frac{\log W}{\varepsilon}$ where c is a constant, the path-selection algorithm selects a subset of paths that survives all of the "-Survivable fibers with
high probability.
\end{lemma}
\begin{proof}
Similar to the argument for the proof of Theorem \ref{Finder-THM}, it is enough to prove Lemma \ref{VC-WDM-bound}.
\end{proof}

\begin{lemma}\label{VC-WDM-bound}
In wavelength restricted version of MSP, VC-dimension $d$ is less than $W$.\\
\end{lemma}
\begin{proof}
Let $X=\{P_{1},...,P_{n}\}$ be the set of all paths, and $R=\{f_{1},...,f_{m}\}$ be the set of all fiber sets where a fiber is associated  i.e., $P_{j} \in f_{i}$ if and only if path $j$ survives fiber $i$'s failure.

Let VC-dimension be $d$. Then, by the definition of VC-dimension, there exist a subset $A$ of paths such that $|A|=d$ and intersection of $A$ with all fiber sets generates all the subsets of $A$. In particular, there exist a fiber set $i$ such that $f_{i} \cap A = \emptyset$, which means $A \subset X-f_i$. Thus,
\begin{align}\label{ineq1}
d=|A| \leq |X-f_i|.
\end{align}

On the other hand, under the wavelength restricted assumption, each fiber can be used by at most $W$ paths. Therefore,
\begin{align}\label{ineq2}
n-W \leq |f_{i}|, \quad \forall f_{i}.
\end{align}

Combining inequalities (\ref{ineq1}) and (\ref{ineq2}) results in $d \leq W$.

\end{proof}

%%%%%%%%%%%%%%%%%%%%%%%%%%%%%%%%%%%%%%%%%%%%%%%%%%%%%
%%%%%%%%%%%%%%%%%%%%%%%%%%%%%%%%%%%%%%%%%%%%%%%%%%%%%

\section{Other ILP formulations of MFSP}\label{other-formulations}
\subsection{Link-Based Formulation}

The idea of Link-Based formulation for MFSP problem is the same as formulation for MSP. The only difference is that we do not need to find the paths in the main logical topology using flow constraints and minimize the number of paths. 

For each fiber $r$, let $f_r$ be a binary variable which takes the value 1
if fiber $r$ is selected, and 0 otherwise. Similar to MSP link-based formulation, variables $x_{ijk}$ refers to the logical links $(i,j)$ and constraint (\ref{flow_remain_MFSP}) refers to the flow constraints in the remaining logical topology $E_L^k$ correspondent to the failure of fiber $k$.

\begin{equation}
\begin{array}{ll}
\mbox{minimize} \quad \quad &  \sum_{r=1}^m f_{r} \quad \quad \quad \quad \quad \quad \quad \quad \quad \quad \quad \quad \quad \quad \quad 
\end{array}
\end{equation}

\begin{equation}
\left.\begin{aligned}\label{flow_remain_MFSP}
\mbox{subject to} &  \sum_{(s,j) \in E_{L}^{k}} x_{sjk} = 1 \quad &\forall \mbox{ Fiber }k\\
\quad &\sum_{(i,t) \in E_{L}^{k}} x_{itk} = 1 \quad &\forall \mbox{ Fiber }k \\
\quad & \sum_{(i,j) \in E_{L}^{k}} x_{ijk} - \sum_{(j,i) \in E_{k}} x_{jik} = 0 \quad &\forall k, \forall i\neq s,t
\end{aligned}
\right\}
\end{equation}

\begin{align}
& f_{r} \geq x_{ijk} & \quad \forall i,j,k, \forall f_r \in x_{ijk} \label{link_fiber_relation} \\
& x_{ijk} \in \{0,1\} & \quad \forall i,j,k
\end{align}

Constraint (\ref{link_fiber_relation}) shows the relation between the selected logical links and fibers, such that fiber $i$ is selected if at least a logical linkusing $f_i$ is selected. The objective function is the sum of selected fibers. Hence, solving this ILP formulation will find a set of survivable paths that uses minimum number of fibers.

\subsection{Cut-Based Formulation}

A set of paths will survive any single failure, if in the occurence of any fiber failure there exist at least one path from $s$ to $t$. 

Let $f_r$ be a binary variable for each fiber $r$, and $y_{ij}$ be a binary variable for each logical link $ij$. Let $N$ be the set of all nodes in the logical topology. The objective function is minimizing the total number of selected fibers.

\begin{align}
\mbox{minimize} & \sum _{r=1}^n f_{r} \\
\mbox{subject to} & \sum _{(i,j) \in E_L^{k}: i \in S, j \in \bar S}y_{ij} \geq 1  \quad \forall k, S \subset N, S \neq N,\emptyset \nonumber \\
& \quad\quad\quad\quad\quad\quad\quad\quad\quad\quad s \in S, t \in \bar S \label{CUT} \\
\quad & \quad\quad f_{r} \geq y_{ij} \quad\quad\quad \forall i,j, \forall f_r \in y_{ij} \label{link_fiber_cut} \\
\quad & \quad\quad y_{ij} \in \{0,1\}
\end{align}

Define ``$s-t$ cut" as a cut $[S,\bar S]$ such that $s \in S$ and $t \in \bar S$. Constraint (\ref{CUT}) shows that for every ``$s-t$ cut" in remaining graph $E_L^k$, there exist at least one logical link from $S$ to $\bar S$. This will gaurantee the survivability of network in the failure of fiber $k$. Constraint (\ref{link_fiber_cut}) builds the relation between selected logical links and fibers, such that a fiber $i$ will be selected if at least one logical link using $f_i$ is selected. Constratints (\ref{CUT}) and (\ref{link_fiber_cut}) select the fibers used by a set of survivable paths and the objective function will find the solution to the MFSP problem.

%%%%%%%%%%%%%%%%%%%%%%%%%%%%%%%%%%%%%%%%%%%%%%%%%%%%%%
%%%%%%%%%%%%%%%%%%%%%%%%%%%%%%%%%%%%%%%%%%%%%%%%%%%%%%

\section{Proof of MFSP complexity}\label{MFSP-complexity-prf-app}
In the physical topology shown in Figure ~\ref{complexity-MFSP-fig}, nodes $s$ and $t$ are the starting and ending nodes. Each node on the left side ($n$ nodes) is connected to 3 nodes on the right such that all the nodes on right are covered by the nodes on left. There are $L$ nodes between $s$ and each node on the left where $L$ is a large number (say $L \geq 3m+3n$  so that the left hand side should be the first priority when minimizing the used number of fibers) and there are $m \geq 3$ nodes on the tail of the graph such that every node on right connects to the tail through the first node $r$.

The logical topology is shown in Figure~\ref{complexity-MFSP-logic-fig}.

\begin{figure}[ht]
	\begin{center}
	\includegraphics[scale=0.4]{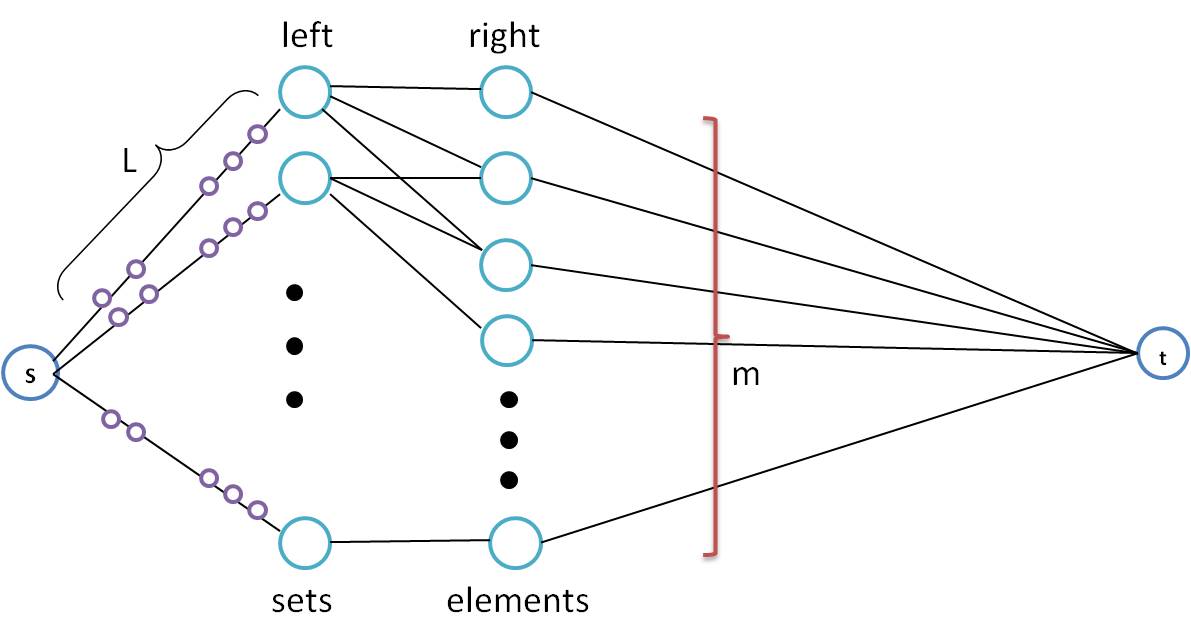}
	\end{center}
	\caption{Logical Topology}
	\label{complexity-MFSP-logic-fig}
\end{figure}

In the logical topology, from $s$ to the nodes on right, each fiber is also a lightpath, while from nodes on right side to t, there are $m$ parallel lightpaths with a specific routing. The first lightpath will be routed on fibers $f_1$, $U_1$ and $L_2$ to $L_m$, the second lightpath will be routed on fibers $f_2$, $L_1$, $U_2$ and $L_3$ to $L_m$ and so on. Therefore, lightpath $i$ will use fibers $f_i$, $U_i$ and all the other $L_j$s ($j \neq i$). \\

To survive any single failure in the firbers from right nodes to node $t$, we need to have at least $m$ paths, each going through one of the parallel logical links. These $m$ paths will not share any fiber from nodes on left to nodes on right, thus any single failure on the fibers between left and right nodes will be survived. Finally, to survive any fiber failure from node $s$ to nodes on the left, it is enoough that at least two of paths use disjoint logical links from node $s$ to left nodes.

Consequently, it is enough just to have $m$ paths covering all nodes on the right hand side. On the other hand, paths between s and left nodes use a large number of fibers. To have a set of paths using the minimum number of fibers, we need to pick the minimum number of nodes from left, to cover all the nodes on right which is a mapping from minimum 3-set cover problem to our problem. The remaining of the proof is explained in the main text.

%%%%%%%%%%%%%%%%%%%%%%%%%%%%%%%%%%%%%%%%%%%%%%%%%%%%%%%
%%%%%%%%%%%%%%%%%%%%%%%%%%%%%%%%%%%%%%%%%%%%%%%%%%%%%%%

\section{Proof of Theorem \ref{cost_greedy}}\label{ACG_PROOF}

Before proving Theorem \ref{cost_greedy}, we need to prove two other Lemmas. Consider a set $S$ of survivable paths and let $F$ be the total number of fibers used by all paths in $S$.

\begin{lemma}\label{path-fiber-thm}
Under the wavelength restricted assumption, the following inequality holds: $\frac{1}{W} \sum_{j=1}^{|S|} C_{j}P_{j} \leq F$
\end{lemma}
\begin{proof}
Figure \ref{figure-path-fiber-W} shows the relation between paths and the fibers used by them. There exist an edge between node $j$ on left ($P_j$) and node $i$ on right ($f_i$) if $f_i \in P_j$. Since each path $j$ is using $C_j$ fibers, the total number of edges is $\sum_{j=1}^{|S|} C_{j}P_{j}$. On the other hand, by assumption, each fiber can be used by at most $W$ paths, therefore each right node can be incident to at most $W$ edges. Thus, we have $\sum_{j=1}^{|S|} C_{j}P_{j} \leq WF$, which completes the proof.

\begin{figure}[h]
\centering
\includegraphics[scale=0.5]{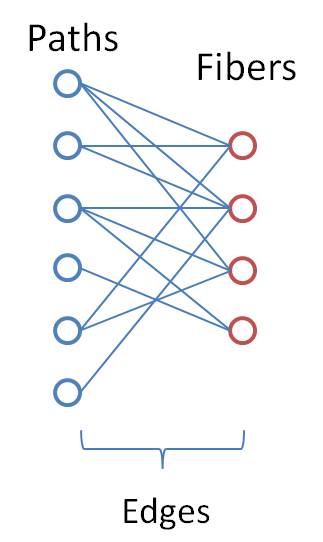}
\caption{relation between paths and fibers }
\label{figure-path-fiber-W}
\end{figure}
\end{proof}

Following theorem is a direct consequence of Lemma \ref{path-fiber-thm}.

\begin{theorem}\label{min-cost-set-cover}
Let $S$ be a set of survivable paths and $R$ be the set of all possible survivable path sets $(S \in R)$. Let $\OPT$ be the minimum number of fibers used by a survivable path set. Then, we have:
$\frac{1}{W} \min_{S \in R} \sum_{j=1}^{|S|} C_{j}P_{j} \leq \OPT \leq \min_{S \in R} \sum_{j=1}^{|S|} C_{j}P_{j}$.
\end{theorem}
\begin{proof}
Lemma \ref{path-fiber-thm} gives the left inequality. In the right inequality, $\min_{S \in R} \sum_{j=1}^{|S|} C_{j}P_{j}$ outputs a set $S$ of survivable paths, therefore it is feasible and gives an upperbound for the optimal solution.
\end{proof}

By Theorem \ref{min-cost-set-cover}, the optimal solution to the problem $\min_{S \in R} \sum_{j=1}^{|S|} C_{j}P_{j}$ provides a $W$ approximation to the MFSP problem. Note that this problem seeks to find a set of survivable paths with minimum cost where the cost of a path is the number of fibers used by that path, and these costs are assumed to be additive. Clearly, this problem is a reduction from minimum cost set cover problem. Since the minimum cost set cover problem is NP-hard, finding a set of survivable paths with minimum additive costs is also NP-hard. Therefore, we use the explained additive cost greedy algorithm to approximate the additive cost survivable path set problem. Now we can prove the $O(W\log m)$ bound stated in Theorem \ref{cost_greedy}.

\begin{proof}
Let $\OPT$ be optimal value of MFSP problem. By the argument in~\cite{Chvatal}, the additive cost of paths selected by ACG is not larger than $O(\log m)\OPT$, i.e.,

\begin{align}
\frac{Greedy}{\log m} \leq \min_{P \in S} \sum_{j=1}^n C_{j}P_{j} \label{greedy_cost}
\end{align}

where $Greedy$ denotes the additive cost of ACG. Combining equation (\ref{greedy_cost}) and Theorem \ref{min-cost-set-cover} gives the following inequality:

\begin{align}
\frac{Greedy}{W \log m} \leq \OPT
\end{align}

\end{proof}

%%%%%%%%%%%%%%%%%%%%%%%%%%%%%%%%%%%%%%%%%%%%%%%%%%%%%%%%%%%%%%%
%%%%%%%%%%%%%%%%%%%%%%%%%%%%%%%%%%%%%%%%%%%%%%%%%%%%%%%%%%%%%%%

\section{Epsilon-Net in MFSP}\label{epsnet_MFSP}

Combining the results of Randomized algorithm described in subsection \ref{sub_K} and Theorem \ref{min-cost-set-cover} results in the following corollary.

\begin{corollary}
Using the $\varepsilon$-net algirithm, one can find an $O(W\log W \log \OPT)$ approximation.
\end{corollary}
\begin{proof}
Similar to the proof of Theorem \ref{cost_greedy}.
\end{proof}

\end{document}